\documentclass[a4paper,12pt]{article}
\usepackage{amsmath,amssymb,amsthm,graphics,latexsym,amsfonts}
\usepackage{fancyhdr}
\usepackage{color}
\usepackage{graphics}
\usepackage{epsfig}
\usepackage{epstopdf}
\renewcommand{\figurename}

\newtheorem{theorem}{Theorem}[section]
\newtheorem{lemma}[theorem]{Lemma}
\newtheorem{corollary}[theorem]{Corollary}

\newtheorem{conjecture}[theorem]{Conjecture}

\usepackage{indentfirst}
\usepackage{geometry}

\begin{document}
\title{\Large Gallai's path decomposition conjecture for Cartesian product of graphs\footnote{Research supported by NSFC (No.12061073)}}

\author{ {Xiaohong Chen\footnote{Email: xhongchen0511@163.com (X. Chen)}, Baoyindureng Wu\footnote{Corresponding author.
Email: baoywu@163.com (B. Wu) }  }\\
\small  \small College of Mathematics and System Science, Xinjiang
University \\ \small  Urumqi 830046, P.R.China \\}
\date{}

\maketitle

{\small \noindent{\bfseries Abstract}:
Let $G$ be a graph of order $n$. A path decomposition $\mathcal{P}$ of $G$ is a collection of edge-disjoint paths that covers all the edges of $G$. Let $p(G)$ denote the minimum number of paths
needed in a path decomposition of $G$. Gallai conjectured that if $G$ is connected, then $p(G)\leq \lceil\frac{n}{2}\rceil$. Let $n_o(G)$ to denote the number of vertices with odd degree in $G$.
Lov\'{a}sz proved that if $G$ is a connected graph with all vertices having degree odd, i.e. $n_o(G)=n$, then $p(G)=\frac n 2$.

In this paper, we prove that if $G$ is a connected graph of order $m\geq 2$ with $p(G)=\frac{n_o(G)}{2}$ and $H$ is a connected graph of order $n$, then $p(G\Box H)\leq\frac{mn}{2}$. Furthermore, we prove that $p(G)=\frac{n_o(G)}{2}$, if one of the following is hold:

(\romannumeral1) $G$ is a tree;

(\romannumeral2) $G=P_n\Box T$, where $n\geq 4$ and $T$ is a tree;

(\romannumeral3) $G=P_n\Box H$, where $H$ is an even graph.
\\{\bfseries Keywords}: Path; Decomposition; Gallai's conjecture; Cartesian product

\section {\large Introduction}

All graphs considered here are finite, undirected and simple. We refer to \cite{Bondy} for unexplained terminology and notation. Let $G=(V(G), E(G))$ be a graph. The order $|V(G)|$ and size $|E(G)|$ are denoted by $n(G)$ and $e(G)$, respectively. The degree and the neighborhood of a vertex $v$ are denoted by $d_G(v)$ and $N_G(v)$, respectively.  The maximum degree of $G$ is denoted by $\Delta(G)$. A vertex is called {\it odd} or {\it even} depending on whether its degree is odd or even, respectively. A graph in which every vertex is odd or even is called an {\it odd graph} or an {\it even graph}. The number of odd vertices and even vertices of $G$ are denoted by $n_o(G)$ and $n_e(G)$, respectively. As usual, we use $P_n$ to denote the path of order $n$.

A {\it path decomposition $\mathcal{P}$} of a graph $G$ is a collection of edge-disjoint paths that covers all the edges of $G$. We use $p(G)$ to denote the minimum number of paths needed for a path decomposition of $G$. Erd\H{o}s asked that what is the minimum number of paths into which every connected graph can be decomposed. As a response to the question of Erd\H{o}s, Gallai made the following conjecture:

\begin{conjecture} [\cite{Lovasz1968}] If $G$ is a connected graph of order $n$, then $G$ can be decomposed into at most $\lceil\frac{n}{2}\rceil$ paths.
\end{conjecture}

The assertion of Gallai's conjecture is sharp if it is true: If $G$ is a graph in which every vertex has odd degree, then in any path decomposition of $G$ each vertex must be the end vertex of some path, and so at least $\lceil\frac{n}{2}\rceil$ paths are required. In 1968, Lov\'{a}sz \cite{Lovasz1968} proved the following theorem and its corollary.

\begin{theorem} [\cite{Lovasz1968}] Every graph of order $n$ can be decomposed into at most $\frac{n}{2}$ paths and cycles.
\end{theorem}

Clearly, $p(G)\geq \max\{\frac{n_o(G)} 2, \frac{\Delta(G)} 2 \}$ for any graph $G$.

\begin{corollary} [\cite{Lovasz1968}] Every odd graph of order $n$ can be decomposed into $\frac{n}{2}$ paths.
\end{corollary}

In 1980, Donald \cite{Donald1980} showed that if $G$ is allowed to be disconnected, then $p(G)\leq \lfloor\frac{3n}{4}\rfloor$, which was improved to $\lfloor\frac{2n}{3}\rfloor$
independently by Dean and Kouider \cite{Dean2000} and Yan \cite{Yan1998}. Furthermore, the Conjecture 1.1 was verified for several classes of graphs. Pyber \cite{Pyber1996} extended Lov\'{a}sz's
result by proving that Conjecture 1.1 holds for graphs whose each cycle contains a vertex of odd degree. Fan \cite{Fan2005} proved that the conjecture is true if each block of the even subgraph is a triangle-free graph with maximum degree at most
three. Botler and Sambinelli \cite{Botler2020(1)} generalized Fan's result in \cite{Fan2005}.
Also, Botler and
Jim\'{e}nez \cite{Botler2017} showed Conjecture 1.1 is ture for a family of even regular graphs with a high girth condition. Harding and McGuinness\cite{Harding2014} proved that for any simple graph $G$ having girth $g\geq 4$, there is a path decomposition of $G$ having at most $\frac{n_o(G)}{2}+\lfloor(\frac{g+1}{2g})n_e(G)\rfloor$ paths.
Recently, Chu, Fan and Zhou \cite{Chu2022} proved that for any triangle-free graph $G$, $p(G)\leq \lfloor\frac{3n}{5}\rfloor$. More known results for Conjecture 1.1 can be found in \cite{Bonamy2019, Botler2019, Botler2020(2), Chu2021, Fan2005, Geng2015, Jimenez2017}.
But in general, it is still open.

The Cartesian product of simple graphs $G$ and $H$ is the graph $G\Box H$ whose vertex set is $V(G)\times V(H)$ and whose edge set is the set of all pairs $(u_{1} ,v_{1})(u_{2} ,v_{2})$
such that either $u_{1}u_{2}\in E(G)$ and $v_{1}=v_{2}$, or $v_{1}v_{2}\in E(H)$ and $u_{1}=u_{2}$. Every connected graph has a unique factorization under this graph product \cite{Sabidussi1960}, and this factorization can be found in linear time and space \cite{Imrich2007}. In 2018, Oyewumi, Akwu and Azer \cite{Oyewumi2018} determined the number of paths needed for a path decomposition of cartesian product of path and cycle. In 2016, Jeevadoss and Muthusamy \cite{Jeevadoss2016} gave a necessary and sufficient condition for the existence of $\{P_5, C_4\}_{\{p,q\}}$-decomposition of Cartesian product of complete graphs, where $G$ has a $\{H_1, H_2\}_{\{p,q\}}$-decomposition means that for all values of $p$ and $q$ satisfying trivial necessary conditions, $G$ has a decomposition into $p$ copies of $H_1$ and $q$ copies of $H_2$.
In 2021, Ezhilarasi and Muthusamy \cite{Ezhilarasi2021} provided a necessary and sufficient conditions for the existence of a complete $\{P_5, S_5\}$-decomposition of Cartesian product of complete graphs.
In recent years, there are many interesting studies on Cartesian product graphs, and some related results can be found in \cite{Anderson2022, Hammack2011, Imrich2008, Kaul2023, Rall2023}.


In this paper, we prove that if $p(G)=\frac{n_o(G)}{2}$, then $p(G\Box H)\leq\frac{mn}{2}$, where $G$ is a graph of order $m\geq 2$ and $H$ is a graph of order $n$. Furthermore, we prove that $p(G)=\frac{n_o(G)}{2}$, if one of the following is hold:
(\romannumeral1) $G$ is a tree;
(\romannumeral2) $G=P_n\Box T$, where $n\geq 4$ and $T$ is a tree;
(\romannumeral3) $G=P_n\Box H$, where $H$ is an even graph.


The paper is organized as follows. In Section 2, we present some definitions,
notation and state an auxiliary result needed in the proof of our main results. In Section 3, we prove that Conjecture 1.1 holds for the Cartesian product of a general graph and a path.
In Sections 4 and 5, we prove our main results.

\section{\large Notation and auxiliary results}

At the beginning of this section, we will first introduce some notations to be used in this paper. Let $S\subseteq V(G)$. The subgraph induced by $S$, denoted by $G[S]$, is the graph with $S$ as its vertex set, in which two vertices are adjacent if and only if they are adjacent in $G$.  The {\it union} of simple graphs $G$ and $H$ is the graph $G\cup H$ with vertex set $V(G) \cup V(H)$ and edge set $E(G) \cup E(H)$. We use $G+uv$ to denote the graph that arises from $G$ by adding an edge $uv\notin E(G)$, where $u, v \in V(G)$. Similarly, $G-u$ is a graph that arises from $G$ by deleting the vertex $u \in V(G)$.

Let $u,v\in V(G)$. The path with ends $u$ and $v$ is denoted by $P_{u,v}$. If $P_{u,v}$ is subgraph of $G$ and $u$ is odd in $G$, we call that $u$ is an {\it odd end vertex} of $P_{u,v}$, otherwise, it is an {\it even end vertex} of $P_{u,v}$. If both $u$ and $v$ are odd in $G$, we call $P_{u,v}$ is an {\it odd-odd path}. If both $u$ and $v$ are even in $G$, then $P_{u,v}$ is an {\it even-even path}. If the degree of $u$ and $v$ have different parity, we call $P_{u,v}$ is an {\it odd-even path}. To {\it subdivide} an edge $e$ is to delete $e$, add a new vertex $x$, and join $x$ to the ends of $e$.


Now, we will prove a useful lemma.


\begin{lemma} For any connected graph $H$, there exists a path decomposition $\mathcal{P}(H)$ of the edges of $H$, such that:

(1) each odd vertex is the end vertex of exactly one path of $\mathcal{P}(H)$;

(2) each even vertex is the end vertex of exactly two path of $\mathcal{P}(H)$;

(3) $|\mathcal{P}(H)|=\frac{n_o(H)}{2}+n_e(H)$.
\end{lemma}
\begin{proof}
Let $H'$ be the graph obtained from $H$ by adding a pendent edge to each even vertex. Clearly, $H'$ is an odd graph with $n(H')=n_o(H)+2n_e(H)$. By Corollary 1.3, let $\mathcal{P}(H')$ be a path decomposition  of $H'$ with $p(H')=\frac{n_o(H)}{2}+n_e(H)$. It implies that each vertex of $H'$ is the end vertex of exactly one path of $\mathcal{P}(H')$. Deleting those new edges from paths in $\mathcal{P}(H')$ result in a path decomposition $\mathcal{P}_0(H)$ of $H$.

It is easy to check that for a vertex $v\in V(H)$, there exists exactly one path in $\mathcal{P}_0(H)$ with $v$ as its end if $d_H(v)$ is odd, and otherwise, either there exists no path with $v$ as its end, or there exist exactly two paths in $\mathcal{P}_0(H)$ with $v$ as its end. If $v\in V(H)$ is even and is not an end vertex of any path in $\mathcal{P}_0(H)$, then there must exist a path $P_{x,y}$ in $\mathcal{P}_0(H)$ that passes through $v$ in $H$. Let $P_{x,v}$ and $P_{v,y}$ be subpaths of $P_{x,y}$. By this way, we obtain a new path decomposition $\mathcal{P}(H)$ of $H$ with satisfying (1) and (2), as we promised. It follows further that (3).
\end{proof}

Let $\mathcal{P}(H)$ be a path decomposition of $H$ as given in the previous lemma. It is evident that the $\mathcal{P}(H)$ can be classified into odd-odd, even-even, odd-even (or even-odd) paths respectively.
We denote the number of odd end vertices of odd-odd paths in $\mathcal{P}(H)$ as $n_{o}^{1}(H)$, and the number of odd end vertices of odd-even paths as $n_{o}^{2}(H)$. Clearly, $n_{o}^{1}(H)+n_{o}^{2}(H)=n_o(H)$.

\begin{theorem} For any tree $T$, $p(T)=\frac{n_o(T)}{2}$.
\end{theorem}

\begin{proof} Obviously, $p(T)\geq \frac{n_o(T)}{2}$. Next, it remains to show that
$p(T)\leq \frac{n_o(T)}{2}$. The proof is by induction on the order $n$ of $T$. If $n=1$, then $T=P_1$, and thus $E(T)=\emptyset$ and $p(T)=0=\frac{n_o(T)}{2}$.
Now let $n\geq 2$, and the result is true for any tree of order less than $n$. Take a leaf $u$ of $T$, and set $T'=T-u$. By the induction hypothesis, $T'$ has a path decomposition $\mathcal{P}(T')$ into $\frac{n_o(T')}{2}$ paths.

Consider the neighbor $v$ of $u$ in $T$.
If $d_{T'}(v)$ is odd, then $n_o(T_n)=n_o(T')$. Let $P'$ be the path with $v$ as its end in $\mathcal{P}(T')$. Clearly, $P=P'+vu$ and $(\mathcal{P}(T')\setminus \{P'\})\cup \{P\}$ is a path decomposition of $T$ with $\frac{n_o(T)}{2}$.

If $d_{T'}(v)$ is even, then $n_o(T)=n_o(T')+2$. Since $vu$ is a new path, $\mathcal{P}(T')\cup \{vu\}$ is a path decomposition of $T$ with $\frac{n_o(T)}{2}$.
\end{proof}

\section{\large Cartesian product of a path and a graph}

First we prove that Gallai's conjecture holds for Cartesian product of a path and a general graph. Let $V(P_m)=\{u_1,u_2,\cdots,u_m\}$ and $V(H)=\{v_1,v_2,\cdots,v_{n}\}$. For the graph $G=P_{m}\Box H$, let $U_i=\{(u_i,v_j): j\in \{ 1,2,\cdots,n\}\}$ where $i\in\{1,2,\cdots,m\}$. Let $\mathcal{P}(H)$ be a path decomposition of $H$ as described in Lemma 2.1. Let $\mathcal{P}_i=\{P^i: V(P^i)=\{(u_i,v): v\in V(P), P\in \mathcal{P}(H)\}\}$, where $i\in\{1,2,\cdots,m\}$. Then $|\mathcal{P}_i|=p(H)= \frac{n_o(H)}{2}+n_e(H)$, and each odd vertex of $G[U_i]$ is the end vertex of exactly one path of $\mathcal{P}_i$, each even vertex of $G[U_i]$ is the end vertex of exactly two paths of $\mathcal{P}_i$.

\begin{theorem}\label{3.1}
Let $m$ be an integer $m\geq 2$ and $H$ be a connected graph of order $n$. If $G=P_{m}\Box H$, then
$p(G)\leq \frac{n(G)}{2}$.
\end{theorem}

\begin{proof}

\vspace{4mm}Note that $G[U_i]\cong H$ for each $i\in \{1,2,\cdots,m\}$. For any path $P_{y_1,y_2}\in \mathcal{P}(H)$, we consider the following three cases.

\vspace{2mm}\noindent{\bf Case 1.} $P_{y_1,y_2}$ is an odd-odd path of $H$.

\vspace{2mm}We can find $m$ paths in $G$ as follows (see an example for the case when $m=3$ as shown in Fig. 1):

\vspace{2mm} If $m$ is even, then

$\{\bigcup_{i=1}^{m} P_{y_1,y_2}^i: i \in \{1,2, \cdots, m\}\} \cup \{(u_i,y_1)(u_{i+1},y_1): i \in \{2,4, \cdots, m-2 \} \} \cup \{(u_i,y_2)(u_{i+1},y_2): i \in \{ 1,3, \cdots, m-1\}\}$,

$\{(u_i,y_1)(u_{i+1},y_1)\}$ for each $i \in \{1,3, \cdots, m-1\}$,

$\{(u_i,y_2)(u_{i+1},y_2)\}$ for each $i \in \{2,4, \cdots, m-2\}$.

\vspace{2mm} If $m$ is odd, then

\vspace{2mm}

$\{\bigcup_{i=1}^{m} P_{y_1,y_2}^i: i \in \{1,2, \cdots, m\}\} \cup \{(u_i,y_1)(u_{i+1},y_1): i \in \{2,4, \cdots, m-1 \} \} \cup \{(u_i,y_2)(u_{i+1},y_2): i\in \{1,3, \cdots, m-2\}\}$,

$\{(u_i,y_1)(u_{i+1},y_1)\}$ for each $i \in \{1,3, \cdots, m-2\}$,

$\{(u_i,y_2)(u_{i+1},y_2)\}$ for each $ i \in \{2,4, \cdots, m-1\}$.

\vspace{2mm}\noindent{\bf Case 2.} $P_{y_1,y_2}$ is an odd-even path of $H$.

\vspace{2mm}Then there exists a unique sequence of even-even paths $P_{y_{j},y_{j+1}}\in \mathcal{P}(H)$ for each $j\in \{2,\ldots, t-1\}$ and an even-odd path $P_{y_{t},y_{t+1}}\in \mathcal{P}(H)$. We can find $m+t-1$ paths of $G$ as follows (see an example for the case when $m=3$ as shown in Fig. 2):

\vspace{2mm}\vspace{2mm} If $m$ is even, then

\vspace{2mm}
$\{\bigcup_{i=1}^{m} P_{y_j,y_{j+1}}^i: i \in \{1,2, \cdots, m\}\} \cup \{(u_i,y_j)(u_{i+1},y_j): i\in \{2,4, \cdots, m-2\}\} \cup \{(u_i,y_{j+1})(u_{i+1},y_{j+1}): i\in \{1,3, \cdots, m-1\}\}$ for each $ j\in\{1, \ldots, t\}$,

$\{(u_i,y_1)(u_{i+1},y_1)\}$ for each $ i \in \{1,3, \cdots, m-1\}$,

$\{(u_i,y_{t+1})(u_{i+1},y_{t+1})\}$ for each $ i \in \{2,4, \cdots, m-2\}$.

\vspace{2mm} If $m$ is odd, then

\vspace{2mm}
$\{\bigcup_{i=1}^{m} P_{y_j,y_{j+1}}^i: i \in \{1,2, \cdots, m\}\} \cup \{(u_i,y_j)(u_{i+1},y_j): i\in \{2,4, \cdots, m-1\}\} \cup \{(u_i,y_{j+1})(u_{i+1},y_{j+1}): i\in \{1,3, \cdots, m-2\}\}$ for each $j \in \{1, \ldots, t \}$,

$\{(u_i,y_1)(u_{i+1},y_1)\}$ for each $ i \in \{1,3, \cdots, m-2\}$,

$\{(u_i,y_{t+1})(u_{i+1},y_{t+1})\}$ for each $ i \in \{2,4, \cdots, m-1\}$.

\begin{figure}[h]
\begin{center}
\includegraphics[height=3.5cm]{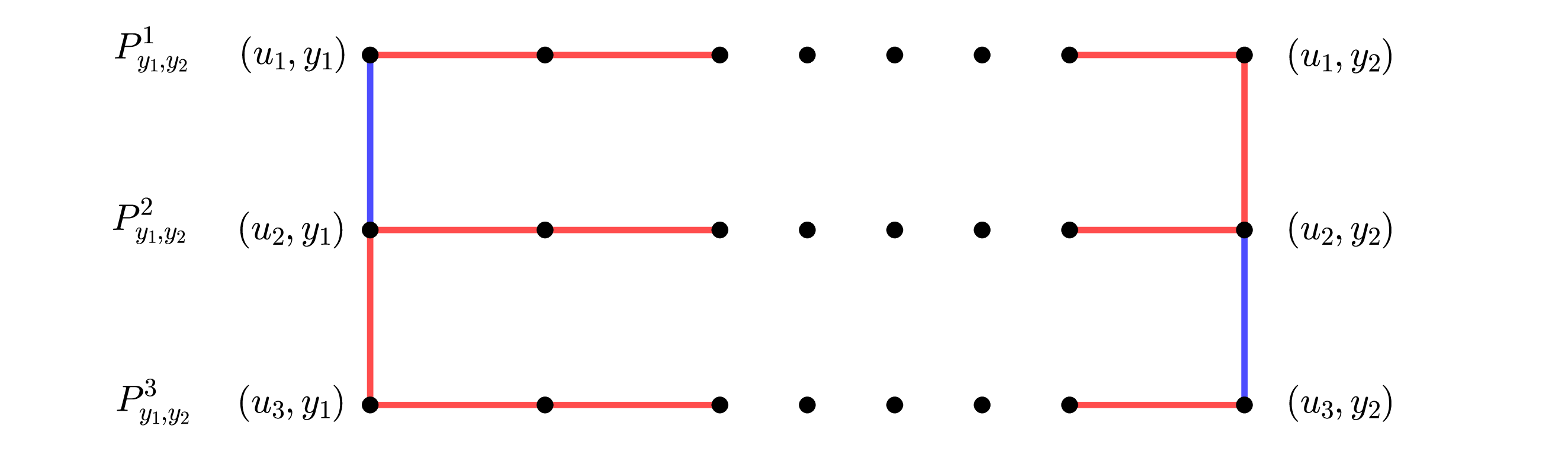}\\
\end{center}
\vspace{2mm}\footnotesize{Fig 1. For $i\in \{1,2,3\}$, each $P_{y_1,y_2}^i$ is an odd-odd path where each vertex has at least one neighbor in $G$.}
\end{figure}

\begin{figure}[h]
\begin{center}
\includegraphics[width=\textwidth]{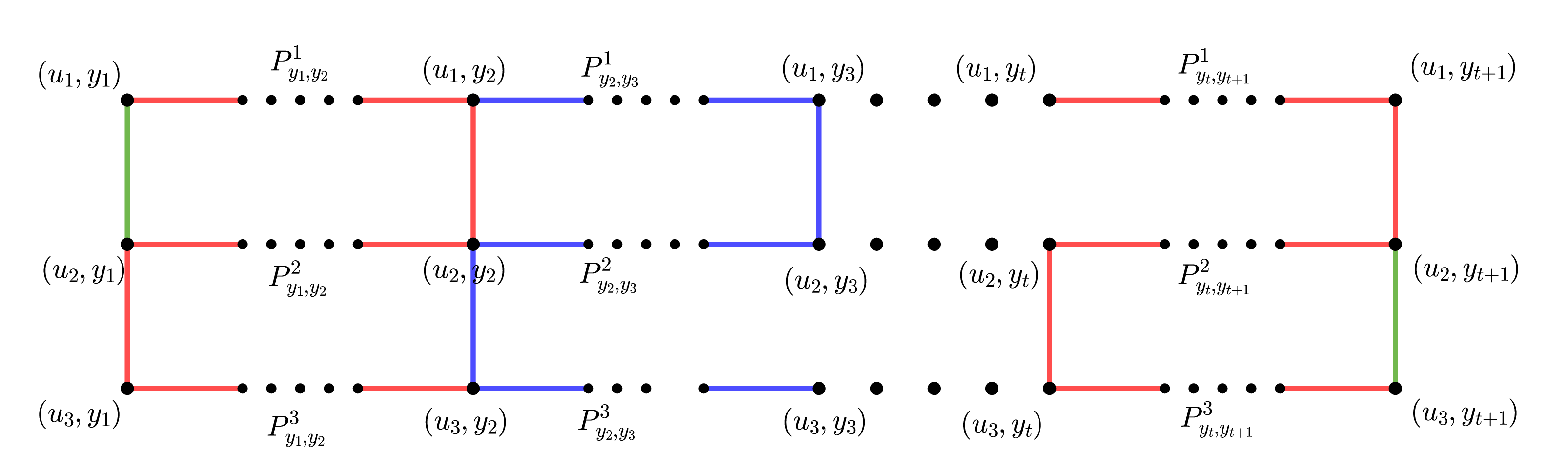}\\
\end{center}
\vspace{2mm}\footnotesize{Fig 2. Each $P_{y_1,y_2}^i$ is an odd-even path, each $P_{y_j,y_{j+1}}^i$ is an even-even path and each $P_{y_t,y_{t+1}}^i$ is an even-odd path, where $i\in \{1,2,3\},j\in \{2,3,\cdots,t-1\}$. Every vertex on the path $P_{y_j,y_{j+1}}^i$ has at least one neighbor in $G$.}
\end{figure}

\begin{figure}[h]
\begin{center}
\includegraphics[width=\textwidth]{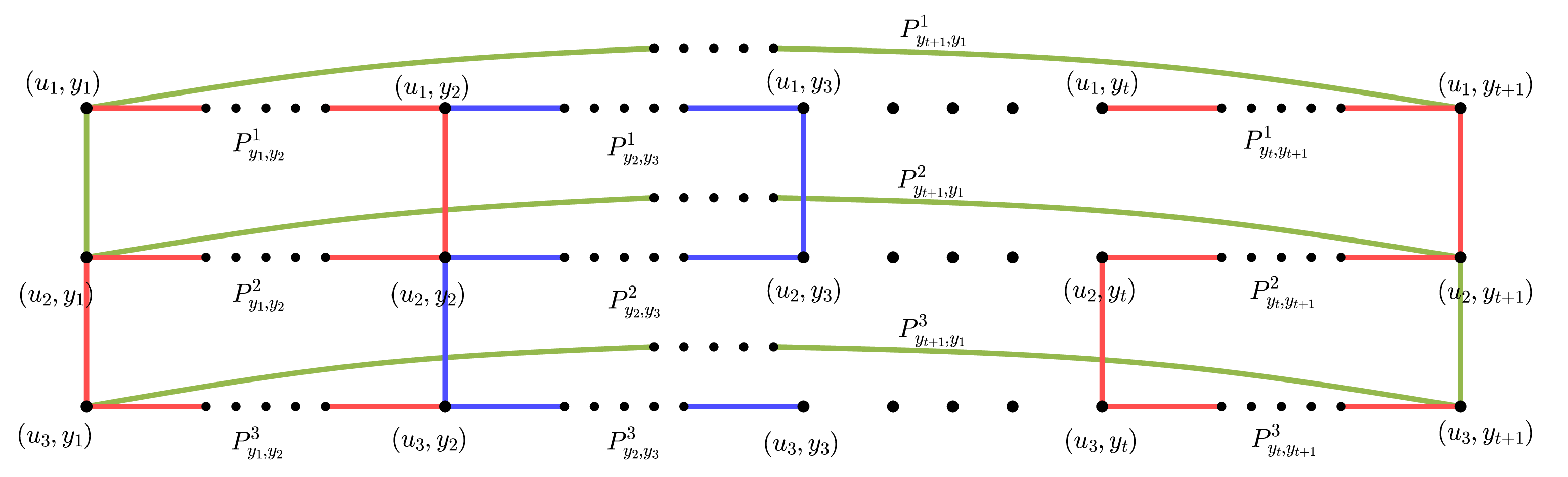}\\
\end{center}
\vspace{2mm}\footnotesize{Fig 3.  Each $P_{y_j,y_{j+1}}^i$ is an even-even path, where $i\in \{1,2,3\}, j\in \{1,2,\cdots,t-1\}$. Every vertex on the path $P_{y_j,y_{j+1}}^i$ has at least two neighbors in $G$.}
\end{figure}

\vspace{3mm}\noindent{\bf Case 3.} $P_{y_1,y_2}$ is an even-even path of $H$.

Then there exists a unique sequence of even-even paths $P_{y_{j}y_{j+1}}\in \mathcal{P}(H)$ for each $j\in \{2,\ldots, t\}$, where $y_{t+1}=y_1$. We can find $t+1$ paths of $G$ as follows (see an example for the case when $m=3$ as shown in Fig. 3):

\vspace{2mm}If $m$ is even, then

\vspace{2mm}$\{\bigcup_{i=1}^{m} P_{y_j,y_{j+1}}^i: i \in \{1,2, \cdots, m\}\} \cup \{(u_i,y_j)(u_{i+1},y_j): i\in \{2,4, \cdots, m-2\}\} \cup \{(u_i,y_{j+1})(u_{i+1},y_{j+1}): i\in \{1,3, \cdots, m-1\}\}$ for each $j\in\{1, \ldots, t\}$.

\vspace{2mm}If $m$ is odd, then

\vspace{2mm}$\{\bigcup_{i=1}^{m} P_{y_j,y_{j+1}}^i: i \in \{1,2, \cdots, m\}\} \cup \{(u_i,y_j)(u_{i+1},y_j): i\in \{2,4, \cdots, m-1 \}\} \cup \{(u_i,y_{j+1})(u_{i+1},y_{j+1}): i\in \{1,3, \cdots, m-2\}\}$ for each $j\in\{1, \ldots, t\}$.

\vspace{2mm}Combining the above, we conclude that there is a path decomposition $\mathcal{P}(G)$ of graph $G$ with $p(G)\leq p(H)+(m-1)\frac{n_{o}^{1}(H)}{2}+(m-1)\frac{n_{o}^{2}(H)}{2}= m\frac{n_o(H)}{2}+n_e(H)\leq \frac{mn}{2}$.


The proof is now finished.

\end{proof}

\section{\large Cartesian product of two general graphs}

In this section, we discuss the path decomposition of the Cartesian product of two general graphs.
To simplify our statement, we give the following general definitions.

\vspace{1.5mm} Let $P_m$ be a path with $V(P_m)=\{u_1,u_2,\cdots,u_m\}$. Some internal vertices of $P_m$ is called {\it virtual}, and the remaining ones are called {\it real}. Note that the ends of the path must be real. Such a path is called a {\it virtual-real} path. The Cartesian product of a virtual-real path $P_m$ and a graph $H$ is defined as the graph with $V(P_m\boxdot H)=V(P_m)\times V(H)$, in which for any $v\in V(H)$, $\{(u_i, v):\ 1\leq i\leq m\}$ induces $P_m$, for each $i\in\{1, \ldots, m\}$, $\{u_i\}\times V(H)$ induces a graph isomorphic to $H$ if $u_i$ is real, and otherwise, $\{u_i\}\times V(H)$ is an independent set.

\begin{lemma}
If $G'$ be a graph obtained from a graph $G$ by subdividing an edge, then $p(G')=p(G)$.
\end{lemma}

Now we are ready to prove one of our main results.

\begin{theorem} Let $G$ be a connected graph of order $m\geq 2$ and $H$ be a connected graph of order $n$. If $p(G)=\frac{n_o(G)}{2}$, then $p(G\Box H)\leq \frac{mn}{2}$.
\end{theorem}

\begin{proof} Let $\mathcal{P}(G)$ be a path decomposition of $G$ with $|\mathcal{P}(G)|=\frac{n_o(G)}{2}$. Let $v$ be any vertex in $G$. It is assigned to be real on the path in $\mathcal{P}(G)$ with $v$ as an end vertex if $v$ is odd, and is assigned to virtual for the remaining paths containing it; If $v$ is even, it is assigned to be real on unique a (randomly chosen) path in $\mathcal{P}(G)$ containing $v$, and is assigned to virtual for the remaining paths containing it. Let $\mathcal{P}(G)=\{Q_i:\ 1\leq i\leq p(G)\}$. For convenience, let $l_i$ be the order of $Q_i$ and $s_i$ be the number of virtual vertex in $Q_i$ for each $i\in\{1, \ldots, p(G)\}$. Based on the above facts, it follows that $\sum_{i=1}^{p(G)}(l_i-s_i)=m$. It is obvious that $E(G)=\cup_{i=1}^{p(G)} E(Q_i\boxdot H)$. By Lemma 4.1, $p(Q_i\boxdot H)= p(P_{l_i-s_i}\Box H)$.  Moreover, by Theorem 3.1,

\begin{eqnarray*}
p(G\Box H)&\leq& \sum_{i=1}^{p(G)} p(Q_i\boxdot H))\\
&=& \sum_{i=1}^{p(G)} p(P_{l_i-s_i}\Box H) \\
&\leq&\sum_{i=1}^{p(G)}\frac{(l_i-s_i)n}{2}\\
&=& \frac{n}{2}\sum_{i=1}^{p(G)}(l_i-s_i)=\frac{mn}{2}.
\end{eqnarray*}
\end{proof}

In view of the above theorem, it become an interesting question: which graph $G$ satisfies $p(G)=\frac{n_o(G)}{2}?$
Clearly, if $G$ is an odd graph, then $p(G)=\frac{n_o(G)}{2}$. We can immediately obtain the following corollaries by Corollary 1.3, Theorem 2.2 and Theorem 4.2.

\begin{corollary} Let $G$ be a connected graph of order $m\geq 2$ and $H$ be a connected graph of order $n$. If $G$ is an odd graph, then $p(G\Box H)\leq\frac{mn}{2}$.
\end{corollary}

\begin{corollary} Let $T$ be a tree of order $m\geq 2$ and $H$ be a connected graph of order $n$. If $G=T\Box H$, then $p(G)\leq \frac{mn}{2}$.
\end{corollary}

\section{\large More graphs $G$ with $p(G)= \frac{n_o(G)}{2}$}

\begin{lemma} For two positive integers $n$ and $t$ with $\max\{n, t\}\geq 4$, $p(P_n\Box P_t)=\frac{n_o(P_n\Box P_t)}{2}$.
\end{lemma}
\begin{proof} Without loss of generality, $n\geq 4$. The result holds trivially if $t=1$. So, let $t\geq 2$.
Let $G=P_n\Box P_t$, $V(P_n)=\{u_1,u_2,\cdots,u_{n}\}$ and $V(P_t)=\{v_1,v_2,\cdots,v_t\}$.
As we have seen before,
\begin{equation} p(G)\geq max\{\lceil\frac{n_o(G)}{2}\rceil, \lceil\frac{\Delta(G)}{2}\rceil\}.
\end{equation}

\vspace{2mm} It is well-known that $d_G(u_i,v_j)=d_{P_n}(u_i)+d_{P_t}(v_j)\leq 2+(m-1)$, we have $\Delta(G)= 4$ and $n_o(G)=2(n+t-4).$ Combining with the above two facts, it follows that $\Delta(G)\leq n_o(G)$ for $n+t\geq 6$. Therefore, by (1), we have $p(G)\geq n+t-4$.

\vspace{2mm} It remains to show that $p(G)\leq n+t-4$.
If $n$ is even, then $G$ has a path decomposition as follows (see an example for the case when $n=6$ as shown in Fig. 4(a)):

\vspace{1.5mm}\noindent$\bullet$ $P_{(u_2,v_1),(u_{n-1},v_{1})}$: $(u_2,v_1)(u_1,v_1)(u_1,v_2)\cdots(u_1,v_t)(u_2,v_t)\cdots(u_n,v_t)
(u_{n},v_{t-1})\cdots(u_n,v_1)$ $(u_{n-1},v_{1})$.

\vspace{1.5mm}\noindent$\bullet$ $P_{(u_i,v_t),(u_{i+1},v_{t})}$: $(u_i,v_t)(u_{i},v_{t-1})\cdots(u_i,v_1)(u_{i+1},v_1)\cdots(u_{i+1},v_{t-1})
(u_{i+1},v_{t})$ for each $i \in \{2,4,\cdots, n-2\}$.

\vspace{1.5mm}\noindent$\bullet$ $P_{(u_i,v_1),(u_{i+1},v_{1})}$: $(u_i,v_1)(u_{i+1},v_{1})$ for each $i \in \{3,5,\cdots,n-3\}$.

\vspace{1.5mm}\noindent$\bullet$ $P_{(u_1,v_j),(u_n,v_j)}$: $(u_1,v_j)(u_{2},v_j)\cdots(u_n,v_j)$ for each $j \in \{2,3,\cdots,t-1\}$.

\vspace{2mm} Summing up the above, $p(P_n\Box P_t)\leq 1+\frac{n-2}{2}+\frac{n-4}{2}+(t-2)=n+t-4$.

\begin{figure}[h]
\begin{center}
\includegraphics[width=\textwidth]{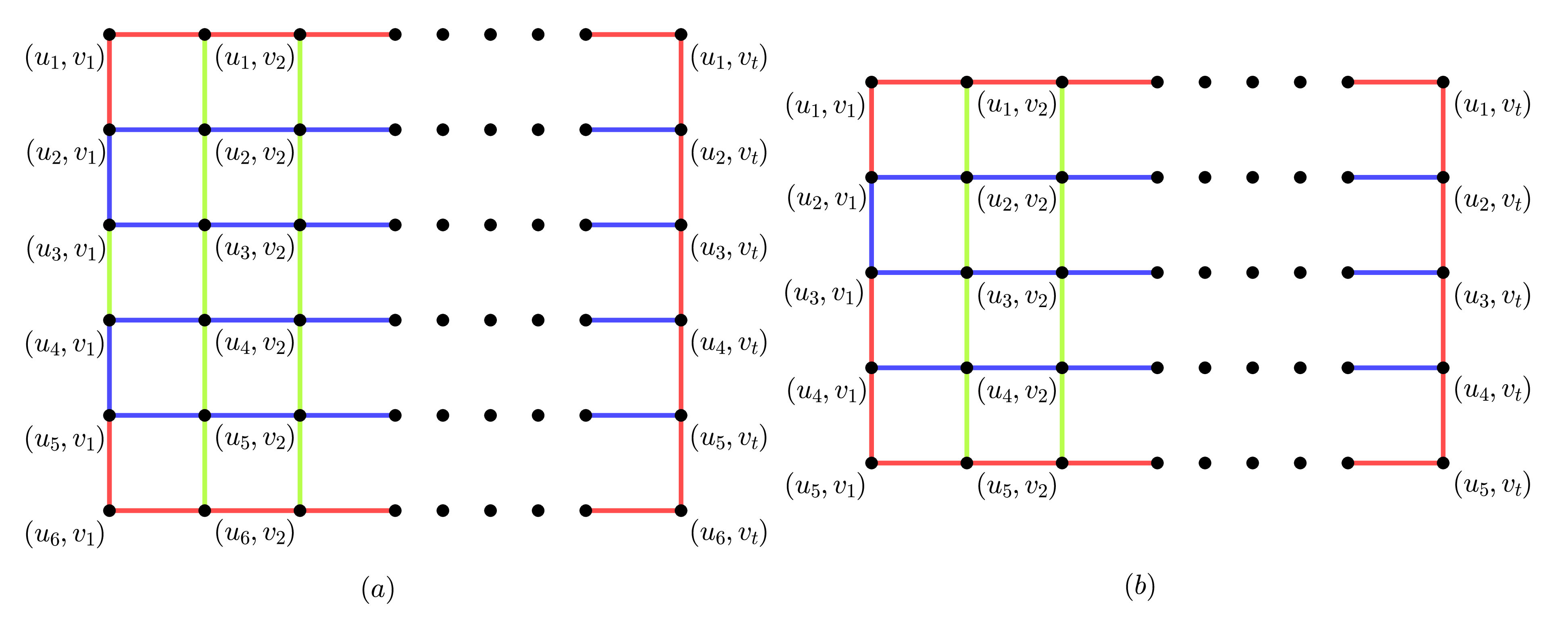}
\end{center}
\vspace{1mm}\footnotesize{Fig 4. (a) is an example of path decomposition of $P_n\Box P_{t}$ when $n=6$ and (b) is an example of path decomposition of $P_n\Box P_{t}$ when $n=5$.}
\end{figure}

\vspace{2mm} If $n$ is odd, then $G$ has a path decomposition as follows (see an example for the case when $n=5$ as shown in Fig. 4(b)):

\vspace{1.5mm}\noindent$\bullet$ $P_{(u_2,v_1),(u_{n-2},v_1)}$: $(u_2,v_1)(u_1,v_1)(u_1,v_2)\cdots(u_1,v_t)(u_2,v_t)\cdots(u_n,v_t)
(u_n,v_{t-1})\cdots(u_n,v_1)$ $(u_{n-1},v_1)(u_{n-2},v_1)$.

\vspace{1.5mm}\noindent$\bullet$ $P_{(u_i,v_t),(u_{i+1},v_{t})}$: $(u_i,v_t)(u_i,v_{t-1})\cdots(u_i,v_1)(u_{i+1},v_1)\cdots(u_{i+1},v_{t-1})
(u_{i+1},v_{t})$ for each $ i \in \{2,4,\cdots,n-3\}$.

\vspace{1.5mm}\noindent$\bullet$ $P_{(u_i,v_1),(u_{i+1},v_{1})}$: $(u_i,v_1)(u_{i+1},v_{1})$ for each $i \in \{3,5,\cdots,n-4\}$.

\vspace{1.5mm}\noindent$\bullet$ $P_{(u_{n-1},v_1),(u_{n-1},v_{t})}$: $(u_{n-1},v_1)(u_{n-1},v_{2})\cdots(u_{n-1},v_{t})$.

\vspace{1.5mm}\noindent$\bullet$ $P_{(u_1,v_j),(u_n,v_j)}$: $(u_1,v_j)(u_{2},v_j)\cdots(u_n,v_j)$ for each  $j \in \{2,3,\cdots,t-1\}$.

\vspace{2mm}Hence, $p(P_n\Box P_t)\leq 1+\frac{n-3}{2}+\frac{n-5}{2}+1+(t-2)=n+t-4$.

The proof is now finished.

\end{proof}

\begin{theorem}
Let $T$ be a tree. For any positive integer $n\geq 4$, $p(P_n\Box T)=\frac{n_o(P_n\Box T)}{2}$.
\end{theorem}

\begin{proof} Let $G=P_n\Box T$, $V(P_n)=\{u_1,u_2,\cdots,u_{n}\}$ and $V(T)=\{v_1,v_2,\cdots,v_m\}$.
As we have seen before, $$p(G)\geq max\{\lceil\frac{n_o(G)}{2}\rceil, \lceil\frac{\Delta(G)}{2}\rceil\}.$$

\vspace{2mm} Since $d_G(u_i,v_j)=d_{P_n}(u_i)+d_{T}(v_j)\leq 2+(m-1)$, we have $\Delta(G)\leq m+1$ and $$n_o(G)=2(m-n_o(T))+(n-2)n_o(T)=2m+(n-4)n_o(T).$$ Based on the above two facts, by a simple calculation it yields $\Delta(G)\leq n_o(G)$ for $n\geq 4$. Therefore, $p(G)\geq m+\frac{(n-4)n_o(T)}{2}$.

\vspace{2mm} It remains to show that $$p(G)\leq m+\frac{(n-4)n_o(T)}{2}.$$

Let $\mathcal{P}(T)=\{Q_i:\ 1\leq i\leq \frac{n_o(T)} 2\}$ be a decomposition of $T$ as given in the beginning of the proof of Theorem 4.2, where each vertex of $T$ assigned as a real vertex of a path in $\mathcal{P}(T)$ exactly once, and a virtual vertex in the other appearance. Similarly, let $Q_i$ has $l_i$ vertices and $s_i$ virtual vertices for each $i$. Then $E(T)=\cup_{i=1}^{p(T)} E(Q_i\boxdot P_n)$. Based on the above facts, it follows that $\sum_{i=1}^{p(T)}(l_i-s_i)=m$. By Lemma 4.1 and Lemma 5.1, $p(Q_i\boxdot P_n)= p(P_{l_i-s_i}\Box P_n)$ and $p(P_n\Box P_{l_i-s_i})=n+l_i-s_i-4$ where $n\geq 4$. Moreover, by Theorem 3.1,
{\footnotesize
\begin{eqnarray*}
p(G)&\leq& \sum_{i=1}^{p(T)} p(Q_i\boxdot P_n))\\
&=& \sum_{i=1}^{p(T)} p(P_{l_i-s_i}\Box P_n) \\
&=&\sum_{i=1}^{p(T)}n+l_i-s_i-4\\
&=&p(T)(n-4)+\sum_{i=1}^{p(T)}(l_i-s_i)\\
&=&m+\frac{(n-4)n_o(T)}{2}.
\end{eqnarray*}}
\end{proof}

\begin{theorem} Let $H$ be an even graph of order $n$. For any positive integer $m\geq 2$, $p(P_m \Box H)= \frac{n_o(P_m \Box H)}{2}$.
\end{theorem}

\begin{proof} Let $G=P_m \Box H$, $V(P_m)=\{u_1,u_2,\cdots,u_{m}\}$ and $V(T)=\{v_1,v_2,\cdots,v_n\}$. As we have seen before,
$$p(G)\geq max\{\lceil\frac{n_o(G)}{2}\rceil, \lceil\frac{\Delta(G)}{2}\rceil\}.$$

Since $\Delta(G)= n+1$ and $n_o(G)=2n.$ we have $\Delta(G)\leq n_o(G)$. Therefore, $p(G)\geq n$.

Next we show that $p(G)\leq n$. Since $H$ is an even graph, $n_o(H)=0$. By Lemma 2.1, there exists a path decomposition $\mathcal{P}(H)$ of the edges of $H$ with $|\mathcal{P}(H)|=n$ such that each vertex is the end vertex of exactly two paths of $\mathcal{P}(H)$. By a similar argument (Case 3) as in the proof of Theorem 3.1, we can find a decomposition $\mathcal{P}(G)$ of $G$ with $|\mathcal{P}(G)|=n$.

\vspace{2mm} This proves $p(G)=n$.
\end{proof}

\noindent {\bf Data availability statement}

\vspace{2mm}\noindent Data sharing not applicable to this article as no data sets were generated or analysed during the current study.

\vspace{3mm}\noindent {\bf Declarations of competing interest}

The authors declare that they have no known competing financial interests or personal relationships that could have appeared to influence the work reported in this paper.

\end{document}